\documentclass[10pt]{amsart}

\usepackage{amsmath}
\usepackage{amsfonts}
\usepackage{amssymb}
\usepackage[latin2]{inputenc}
\usepackage{slashbox}

\newtheorem{defin}{Definition}
\newtheorem{theorem}{Theorem}
\newtheorem{conjecture}{Conjecture}
\newtheorem{corollary}{Corollary}
\newtheorem{lemma}{Lemma}
\newtheorem{obs}[lemma]{Observation}

\def\stirling#1#2{\genfrac{\{}{\}}{0pt}{}{#1}{#2}}

\def\set#1{\{#1\}}

\begin{document}

\author{Be\'ata B\'enyi}
\address[Be\'ata B\'enyi]{J\'ozsef E\"otv\"os College\\
            Bajcsy-Zsilinszky u. 14., Baja, Hungary 6500}

\author{P\'eter Hajnal}
\address[P\'eter Hajnal]{Bolyai Institute, University of Szeged\\
            Aradi v\'ertan\'uk tere 1, Szeged, Hungary 6720\\
            and\\
            Alfr\'ed R\'enyi Institute of Mathematics\\
            Hungarian Academy of Sciences\\
            13--15 Re\'altanoda utca, 1053 Budapest, Hungary}

\title{Combinatorial properties of poly-Bernoulli relatives}

\begin{abstract}
In this note we augment the poly-Bernoulli family with two new
combinatorial objects. We derive formulas for the relatives of the
poly-Bernoulli numbers using the appropriate variations of
combinatorial interpretations. Our goal is to show connections between
the different areas where poly-Bernoulli numbers and their relatives
appear and give examples how the combinatorial methods can be used for
deriving formulas between integer arrays.
\end{abstract}

\maketitle


\section{Introduction}

Poly-Bernoulli numbers were introduced by M. Kaneko \cite{KanekoIntro}
in 1997 as a generalization of the classical Bernoulli numbers during
his investigations of multiple zeta values.  The sequence received
attention because of its nice properties, that were proved by several
authors analytically.  The importance of the notion of the
poly-Bernoulli numbers is underlined also by the fact that there are
several drastically different combinatorial interpretations
\cite{BH}. The combinatorics of the family of poly-Bernoulli numbers
is shown in the bijections that can be described between the
sets. These bijections help us to understand more the properties of
the poly-Bernoulli numbers.

In this paper we consider two number arrays that are relatives of
poly-Bernoulli numbers.  The importance of the attention in this
direction is that in some combinatorial problems these relatives arise
naturally.  Also Kaneko's number theoretical investigations led to
these numbers.  We go through the known combinatorial
interpretations of poly-Bernoulli numbers \cite{BH} and for most of
them we show that slight modifications of the original combinatorial
definition lead to the descriptions of the two related sequences.

This way we connect poly-Bernoulli numbers to the class of
permutations with a special excedance set. We
augment the list of poly-Bernoulli families with two classes of $01$
matrices defined by a given forbidden set of submatrices.

The outline of the paper is as follows. After a short introduction of
the poly-Bernoulli numbers we define the poly--Bernoulli relatives
using the well known interpretation of lonesum matrices.  We derive
different formulas for these arrays and show relations between the
number sequences using appropriate combinatorial interpretations.
We close our discussion with a conjecture related to the central
binomial sum.
 

\subsection{Poly-Bernoulli numbers}

The story of the Bernoulli numbers starts with
investigating the sum of the $m$th
powers of the first $n$ positive integer that are polynomials in
$n$. Jacob Bernoulli recognized the scheme in the coefficients of
these polynomials.
Kaneko generalized the well known generating function 
of the Bernoulli numbers 
and defined the poly-Bernoulli numbers.

\begin{defin}\cite{KanekoAT}
Poly--Bernoulli numbers (denoted by $B_{n}^{(k)}$, where $n$ is a positive integer and $k$ is an
integer) are defined by the following
exponential generating function
\begin{equation}
\sum_{n=0}^{\infty}B_n^{(k)}\frac{x^n}{n!} =\frac{Li_k(1-e^{-x})}{1-e^{-x}},
\end{equation} where
\[
Li_k(z) =\sum_{i=1}^{\infty}\frac{z^i}{i^k}, 
\]
i.e. $Li_k(z)$ is the $k$th poly-logarithm when $k>0$ 
and a rational function when $k\leq 0$.  
\end{defin} 

From the combinatorial point of view we are interested only in the
poly-Bernoulli numbers with negative $k$ indices since in this case
the numbers form an array of positive integers. From now on we mean
poly-Bernoulli numbers always with negative indices even if we don't
emphasize it explicitly. For the sake of convenience we denote in the
rest of the paper $B_n^{(-k)}$ as $B_{n,k}$. The following table shows
the values of poly-Bernoulli numbers for small indices.  An extended
array can be find in OEIS \cite{OEIS} A099594.

\begin{table}[h]
\caption{The poly-Bernoulli numbers $B_{n,k}$}
\begin{center}
\begin{tabular}{|c||c|c|c|c|c|c|}
\hline
\backslashbox{$n$\kern-2.5em}{\kern-.5em $k$} 
  & 0 & 1 & 2 & 3 & 4 & 5\\
\hline\hline
0 & 1 & 1& 1& 1 & 1 & 1\\
\hline
1 & 1 & 2 & 4 & 8 & 16 & 32\\
\hline
2 & 1 & 4 & 14 & 46 & 146 & 454 \\
\hline
3 & 1 & 8 & 46 & 230 & 1066 & 4718 \\
\hline
4 & 1 & 16 & 146 & 1066 & 6906 & 41506\\
\hline
5 & 1 & 32 & 454 & 4718 & 41506 & 329462\\
\hline
\end{tabular}
\end{center}
\end{table}

The symmetry of the array in $n$ and $k$ is immediately conspicuous.
Analytically this property is obvious from the symmetry of the double
exponential function:

\begin{align}
\sum_{k=0}^{\infty}\sum_{n=0}^{\infty}B_{n,k}\frac{x^n}{n!}\frac{y^k}{k!} =
\frac{e^{x+y}}{e^x + e^y -e^{x+y}}.
\end{align}

Three formulas of poly-Bernoulli numbers were proved combinatorially in the literature:

the combinatorial formula (\cite{Brewbaker}, \cite{BH})

\begin{equation}
B_{n,k} =  \sum_{m=0}^{\min(n,k)}m!\stirling{n+1}{m+1}m!\stirling{k+1}{m+1},
\end{equation}

an inclusion-exclusion type formula (\cite{Brewbaker})

\begin{align}\label{eq:szita}
B_{n,k}=(-1)^n\sum_{m=0}^n(-1)^m m!\stirling{n}{m}(m+1)^k, 
\end{align}

and a recursion (\cite{BH})

\begin{align}\label{eq:rec}
B_{n,k+1} = B_{n,k} + \sum_{m=1}^n\binom{n}{m}B_{n-(m-1),k}.
\end{align}

One of the first (and widely known) combinatorial interpretation of
the poly-Bernoulli numbers are lonesum matrices
\cite{Brewbaker}. Lonesum matrices arise in the roots of discrete
tomography.  Ryser \cite{Ryser} investigated in the late 1950's the
problem of the reconstruction of a matrix from given row and column
sums.
The $01$ matrices that are uniquely reconstructible
from their row and column sum vectors are called \emph{lonesum}
matrices. We denote the set of lonesum matrices of size $n\times k$
as $\mathcal L_n^{k}$.
Note that we allow $n=0$ (and $k=0$ too), in which case
the empty matrix counted as lonesum.


\begin{theorem}\cite{Brewbaker}
The number of $01$ lonesum matrices of size $n\times k$ 
is given by the poly-Bernoulli numbers 
of negative $k$ indices. 
\[
|\mathcal{L}_n^{k}|= 
\sum_{m=0}^{\min(n,k)}(m!)^2\stirling{n+1}{m+1}\stirling{k+1}{m+1}= B_{n,k}.
\]
\end{theorem} 

\begin{proof}(Sketch)
Take a lonesum matrix $M$ of size $n\times k$.
Add a new column and new row with all $0$ entries and
obtain lonesum matrix $\widehat M$ of size $(n+1)\times(k+1)$.
We know that $\widehat M$ contains at least one all-$0$ row and
at least one all-$0$ column (this information was not known for $M$).
Partition the rows and the columns according to the sum of its entries.
In the case of lonesum matrices `having the same row/column sum'
and `being equal' is the same relation. Easy to see that the
number of row classes will be the same as the number of
equivalence classes of columns. We denote this common value by $m+1$.
The plus $1$ stands for the class of extra row/column,
the class of all-$0$ rows and all-$0$ columns.
The row sums order the ($m$ many) classes of not all-$0$ rows.
Similarly the column sums order the classes of not all-$0$ columns.
Our formula comes from the fact that from the two partitions and two orders
it is easy to decode $M$.
\end{proof}

This important theorem started the combinatorial investigations
of poly-Bernoulli numbers. 
From the point of combinatorics $B_{n,k}=|\mathcal L_n^{k}|$
is the natural way of defining the poly-Bernoulli numbers.
It turned out that there are several alternative
combinatorial ways to describe the poly-Bernoulli numbers.
Some of them were investigated before Kaneko's pioneering
work. Next we define two related
$2$-dimensional sequences combinatorially.


\subsection{PB-Relatives}

We consider lonesum matrices with further restrictions on the
occurrence of all-$0$ columns resp.~all-$0$ rows.  More precisely let
$\mathcal{L}_n^k(c|)$ 
denote the set of lonesum matrices with the
property that each column contains at least one $1$ entry and
$\mathcal{L}_n^k(c|r|)$ the set of lonesum matrices with the property
that each column and each row contains at least one $1$. 

\begin{defin}
$C_{n,k}$ denotes $|\mathcal{L}_n^k(c|)|$, i.e.~the number
of lonesum matrices of $n\times k$ without all-$0$ columns.

$D_{n,k}$ denotes $|\mathcal{L}_n^k(c|r|)|$, i.e.~the number
of lonesum matrices of $n\times k$ without all-$0$ columns and all-$0$ rows.
\end{defin}
 
Let us see the first few values of our new numbers:

\begin{table}[h]
\caption{ Poly--Bernoulli relatives: $C_{n,k}$ and  $D_{n,k}$}
\begin{center}
\begin{tabular}{|c||c|c|c|c|c|}
\hline
\backslashbox{$n$\kern-2.5em}{\kern-.5em $k$} 
& 0 & 1 & 2 & 3 & 4 \\
\hline\hline
1  & 1 & 1 & 1& 1& 1\\
\hline
2   & 1 & 3 & 7 & 15 & 31\\
\hline
3  & 1 & 7 & 31 & 115 & 391\\
\hline
4 & 1 & 15 & 115 & 675 & 3451\\
\hline
5  & 1 & 31 & 391 & 3451 & 25231\\
\hline
\end{tabular}
\quad\quad 
\begin{tabular}{|c||c|c|c|c|c|}
\hline
\backslashbox{$n$\kern-2.5em}{\kern-.5em $k$} 
& 1 & 2 & 3 & 4 & 5 \\
\hline\hline
1 & 1 & 1 & 1& 1& 1\\
\hline
2 &1  & 5 & 13 & 29 & 61  \\
\hline
3 & 1 & 13 & 73 & 301 & 1081 \\
\hline
4& 1 & 29 & 301 & 2069 & 11581  \\
\hline
5 & 1 & 61 & 1081 & 11581 & 95401\\
\hline
\end{tabular}
\end{center}
\end{table}

First we give combinatorial formulas for the two new numbers:

\begin{theorem}  We have
\begin{itemize}
\item[(i)] for $\quad n\geq 1\text{ and }k\geq 0$
\[
C_{n,k}=|\mathcal{L}_n^k(c|)| = 
\sum_{m=0}^{\min(n,k)}(m!)^2\stirling{n+1}{m+1}\stirling{k}{m},
\]
\item[(ii)] for $\quad n\geq 1\text{ and }k\geq 1$
\[
D_{n,k}=|\mathcal{L}_n^k(c|r|)| =
\sum_{m=0}^{\min(n,k)}(m!)^2\stirling{n}{m}\stirling{k}{m}
.
\]
\end{itemize}
\end{theorem}

\begin{proof}(Sketch)
(i): We have the information that there is no all-$0$ column.
Hence we do not need the extra column. The extra row ensures
that the extended matrix has the class
of all-$0$ rows. $m$ denotes the number of classes of
($k$ many non-$0$) columns. $m+1$ will be the
number of classes of the $n+1$ rows.
The rest is a straightforward repeat of the original argument.

(ii) is immediate by the same logic.
\end{proof}

From the combinatorial definition it is obvious that the series
$B_{n,k}$ and $D_{n,k}$ are symmetric in $n$ and $k$. The symmetry of
the $C_{n,k}$ numbers ($C_{n,k}=C_{k+1,n-1}$)
is also transparent from our table. But its proof
is not straightforward. We present it in a latter section.
 
The combinatorial definitions make it clear that the sequence
$C_{n,k}$ is the binomial transform of $D_{n,k}$ and the sequence
$B_{n,k}$ is the binomial transform of $C_{n,k}$. Precisely:

\begin{obs} 
The following relations hold
\begin{itemize}
\item[(i)]
\[
B_{n,k} = 1+\sum_{i=1}^k\binom{k}{i}C_{n,i} = 
\sum_{i=0}^k\binom{k}{i}C_{n,i} , \qquad (k\geq 0, n\geq 1),
\]
\item[(ii)]
\[
C_{n,k} = \sum_{i=1}^{n}\binom{n}{i}D_{i,k}, \qquad (k\geq 1,n\geq 1),
\]
\item[(iii)]
\[
B_{n,k} = 1+
\sum_{i=1}^{n}\sum_{j=1}^{k}\binom{n}{i}\binom{k}{j} D_{i,j}, 
\qquad (k\geq 1, n\geq 1).
\]
\end{itemize}
\end{obs}
  
\begin{proof}
(i): To describe an arbitrary non-$0$ lonesum matrix we need
to identify all its columns with at least one $1$
(their number is denoted by
$i(>0)$) and their entries in these $i$ columns (that is
describing
a lonesum matrix of size $n\times i$ that contains 
at least one $1$ in each column). This simple fact proves 
the formula of (i).

We obtain (ii) following the same argument
on rows. (iii) summarizes (i) and (ii).
\end{proof}

There are other connections, recursions, combinatorial properties
of the poly-Bernoulli numbers and its two relatives.
They are not obvious, we discuss them when
the appropriate combinatorial interpretations
appear below.

These numbers seem to be just minor modifications
of the original poly-Bernoulli numbers.
In spite of the first impression, it turns out that these numbers
appeared a natural way in earlier papers.
Now we summarize the analytical properties of these numbers
(obtained by others). 
The rest of the paper
is combinatorial.


\subsection{Analytical results in the literature}

Arakawa and Kaneko \cite{ArakawaKaneko}
introduced a function that are referred in the literature as the
Arakawa-Kaneko function.
\[
\xi_k(s):= \frac{1}{\Gamma(s)}\int_0^{\infty}\frac{t^{s-1}}{e^t-1}Li_k(1-e^{-t})dt.
\] 
The values of this function at non-positive integers are given by 
\[
\xi_k(-m)= (-1)^mC_m^{(k)},
\]
where the generating function of the numbers $\{C_n^{(k)}\}$ (for
arbitrary integers $k$) is given by
\[
\sum_{n=0}^{\infty}C_n^{(k)}\frac{x^n}{n!} = \frac{Li_k(1-e^{-x})}{e^x-1}.
\]

They computed the double exponential generating function
of the $C_n^{(-k)}$ numbers.
We know that
the exponential functions of two number sequences differ only by an
$e^x$ (resp.~ $e^y$) factor when one sequence is the binomial
transform of the other. From this observation
we can conclude the binomial transformation
relation between poly-Bernoulli
numbers and $\set{C_n^{(-k)}}$. It is immediate that $C_{n,k}=C_n^{(-k)}$. 
Furthermore
we can obtain the
generating function of $D_{n,k}$ numbers.

\begin{theorem}
\begin{itemize}
\item[(i)]
\[
\sum_{n=1}^{\infty}\sum_{k=1}^{\infty}C_{n,k}\frac{x^n}{n!}\frac{y^k}{k!}=
           \frac{e^{x}}{e^x+e^y-e^{x+y}},
\]
\item[(ii)]
\[
\sum_{n=1}^{\infty}\sum_{k=1}^{\infty}D_{n,k}\frac{x^n}{n!}\frac{y^k}{k!}=
           \frac{1}{e^x+e^y-e^{x+y}}.
\]
\end{itemize}
\end{theorem}



Kaneko realized the importance of
the C-relative and in a recent paper \cite{KanekoLast} summarized
formulas and properties of $B_n^{(k)}$ and $C_n^{(k)}$
parallel. Kaneko showed also a simple arithmetic connection between
the two series.
\[
B_{n,k} =C_{n,k}+C_{n+1, k-1}
\]

In our investigations we show this relation combinatorially using the
variations of the so called Callan permutations. Moreover we prove a
similar relation between the series $D_{n,k}$ and $C_{n,k}$.



\section{$01$ matrices with excluded submatrices}

The study of matrices that are characterized by excluded submatrices
is an active research area with many important results and
applications \cite{MT}.
Given two matrices $A$ and $B$ we say that $A$ avoids $B$ whenever $A$
does not contain $B$ as a submatrix.  (Given a matrix $M$ a submatrix
is a matrix that can be obtained from $M$ by deletion of rows and
columns.)

Generally we can set the following problem: Let
$S=\left\{M_1,\ldots,M_r\right\}$ be a set of $01$ matrices.
$\mathcal{M}_n^k(S)$ denote the $n\times k$ $01$ matrices that do
not contain any matrix of the set $S$, $\mathcal{M}_n^k(S;c|)$  denote these matrices with
the extra condition of containing in any column at least one $1$, and 
$\mathcal{M}_n^k(S;r|c|)$ denote those with the same extra condition on rows also. 

Lonesum matrices can be characterized also with the terminology of
forbidden submatrices \cite{Ryser}.  
Lonesum matrices are matrices that avoid the
following set of submatrices:
\[
L = \left\{
\begin{pmatrix} 1 & 0\\0 & 1 \end{pmatrix}, 
\begin{pmatrix}0 & 1 \\ 1 & 0\end{pmatrix}\right\}
\]
I.e.~$\mathcal L_n^k=\mathcal M_n^k(L)$, $\mathcal L_n^k(c|)=\mathcal M_n^k(L;c|)$ and $\mathcal L_n^k(r|c|)=\mathcal M_n^k(L;r|c|)$.

Interestingly beyond the set $L$ there are other matrix sets $S$ which forbiddance as submatrices lead to the
poly-Bernoulli numbers. 


\subsection{$\Gamma$-free matrices and recursions}

In \cite{BH} the authors investigated the so called $\Gamma$-free
matrices, matrices with the forbidden set:
\[
\Gamma =\left\{
 \begin{pmatrix} 1 & 1\\1 & 0 \end{pmatrix}, 
\begin{pmatrix}1 & 1 \\ 1 & 1\end{pmatrix}\right\}
\]
and showed bijectively that the number of $n\times k$ $\Gamma$-free
matrices (their set is denoted by $\mathcal{G}_n^k$) 
are the $B_{n,k}$ poly-Bernoulli
numbers. Clearly the forbiddance of all $0$ rows/resp.  columns has the
same effect in this case as in the case of the lonesum matrices. 

\begin{theorem}We have
\begin{itemize}
\item[(i)]
\[
|\mathcal{G}_n^k(c|)|=C_{n,k},
\]
\item[(ii)]
\[
|\mathcal{G}_n^k(r|c|)|=D_{n,k}.
\]
\end{itemize}
\end{theorem}
 
The structure of these matrices gives a transparent explanation of the
recursive formula of poly-Bernoulli numbers that was first proven by
Kaneko \cite{ArakawaKaneko}. In the same spirit we can establish the
recursions concerning the poly-Bernoulli relatives.

\begin{theorem}
\begin{itemize}
\item[(i)]
\[
C_{n,k+1}=\sum_{m=1}^{n}\binom{n}{m}C_{n-m+1,k},
\]
\item[(ii)]
\[
D_{n,k+1}= \sum_{m=1}^n\binom{n}{m}(D_{n-m,k}+ D_{n-m+1,k}).
\]
\end{itemize}
\end{theorem}

\begin{proof}
(i):
$C_{n,k+1}$ counts the
$\Gamma$-free matrices of size $n\times(k+1)$
without all-$0$ column. Each row of a $\Gamma$-free matrix 
\begin{itemize}
\item[A.] starts with a $0$ or
\item[B.] starts with a $1$ followed only by $0$s or 
\item[C.] starts with a $1$ and contains at least one more $1$.
\end{itemize}
Let $m$ denote the number of rows that starts with a $1$. $m\geq 1$,
since all columns contain at least one $1$.  We choose these $m$ rows
$\binom{n}{m}$ ways.  The first $m-1$ rows has to be of type $B$ since
a $\Gamma$ would appear.  The further $(n-m+1)\times k$ elements can
be filled with an arbitrary $\Gamma$-free matrix that contain in any
column at least one $1$.

(ii): If we argue the same way as before we obtain
\[
\sum_{m=1}^{n}\binom{n}{m}D_{n-m+1,k},
\]
but we do not count matrices that contain only type $A$ and type $B$ rows (and does not have type C rows).
In this case the remainder $(n-m+1)\times k$ elements
contains at least an all-$0$ row (the remainder of a type B row).
Hence these matrices are not counted in the above formula.

To correct the enumeration
(count the missing matrices) we must add the term
\[
\sum_{m=1}^n\binom{n}{m}D_{n-m,k},
\]
and obtain (ii).
\end{proof}


\subsection{Permutation tableaux of size $n\times k$}

Permutation tableaux were introduced by Postnikov \cite{Postnikov}
during his investigations of totally Grassmannian cells.   
They received a lot of attention after \cite{Viennot} Viennot showed its one-to-one correspondence to permutations, alternative tableaux and the strong connection to the PASEP modell in statistical mechanics. Many bijections arised in the literature to other objects (tree-like tableaux) and to permutations in order to are use them 
 for enumerations of permutations according to certain statistics \cite{Corteel}, \cite{Burstein}. 
Permutation tableaux are usually defined as $01$ fillings of
Ferrers diagram with the next two conditions:
\begin{itemize}
\item[(column)]
      each column contains at least one $1$.
\item[(1--hinge)] 
      each cell with a $1$ above in the same column and to its left 
      in the same row must contain a $1$.
\end{itemize}
In the special case when the Ferrers diagram is a $n\times k$ array
the definition gives actually the set $\mathcal{M}_n^k(P; c|)$, where
\[
P = \left\{\begin{pmatrix} 0 & 1 \\ 1 & 0 \end{pmatrix},
\begin{pmatrix}1 & 1\\ 1& 0 \end{pmatrix}\right\}
\]
 
\begin{theorem}We have
\begin{itemize}
\item[(i)]
\[
|\mathcal{M}_n^k(P)|=B_{n,k},
\]
\item[(ii)]
\[
|\mathcal{M}_n^k(P;c|)|=C_{n,k},
\]
\item[(iii)]
\[
|\mathcal{M}_n^k(P;r|c|)|=D_{n,k}.
\]
\end{itemize}
\end{theorem}

\begin{proof}
(i) is contained in \cite{Kitaev} without the recognition
of the relation to the poly--Bernoulli numbers.  In \cite{Yun} in
Lemma 4.3.5 the author proves the formula also and as a corollary he
receives that the number of $n\times k$ patterns of permutation
diagrams is the poly-Bernoulli numbers $B_{n,k}$. For details see
\cite{Yun}.

(ii), (iii) is proved by the obvious
binomial correspondences between
$|\mathcal{M}_n^k(P)|$,
$|\mathcal{M}_n^k(P;c|)|$, and
$|\mathcal{M}_n^k(P;r|c|)|$.
\end{proof}
The theorem follows also from a certain bijection between permutations  
 and permutation tableaux that we cite in a latter section.

We see that in the case of permutation tableaux the important variant
is the $C$-relative, the one that corresponds to the restriction of
the columns.  This is one of the reason why we think that the
introduction and investigation of the variants of poly-Bernoulli numbers
is useful.


\subsection{A further excluded submatrix set}

Brewbaker made extensive computations
considering $01$ matrices with
excluded patterns \cite{Brewbakerpc}. These suggest the following
theorem that we prove by showing the recursion for the
poly-Bernoulli numbers.

\begin{theorem}
Let $Q$ be the set
\[
Q= \left\{\begin{pmatrix} 1 & 1 \\ 1 & 0 \end{pmatrix},
\begin{pmatrix}1 & 0\\ 1& 1 \end{pmatrix}
 \right\}.
\]
Then we have
\begin{itemize}
\item[(i)]
\[
 |\mathcal{M}_n^k(Q)|=B_{n,k},
\]
\item[(ii)]
\[
 |\mathcal{M}_n^k(Q;c|)|=C_{n,k},
\]
\item[(iii)]
\[
|\mathcal{M}_n^k(Q;r|c|)|=D_{n,k}.
\]
\end{itemize}
\end{theorem}

\begin{proof}
(i):
Let $M$ be a matrix in $\mathcal{M}_n^k(Q)$ and
$Q_{n,k}=|\mathcal{M}_n^k(Q)|$. Let $j_1>j_2>\cdots >j_m$ the indices
of the rows of the $1$ entries in the first column ($m\geq 0$). When
$m=0$ or $m=1$ the first row does not restrict the remainder $n\times
(k-1)$ entries hence it can be filled with an arbitrary matrix in
$\mathcal{M}_n^{k-1}(Q)$. When $m\geq 2$ the rows $j_1,j_2,\ldots j_m$
coincide otherwise one of the submatrix in $Q$ would appear.
It is enough to describe a $(n-m+1)\times(k-1)$
$Q$-free matrix in order to define $M$. We have:
\[
Q_{n,k}=(n+1)Q_{n,k-1} + \sum_{m=2}^{n}\binom{n}{m}Q_{n-m+1,k-1}.
\]
Hence the $Q_{n,k}$ numbers and the poly-Bernoulli
numbers
satisfy the same recursion. Induction proves (i).

(ii), (iii): The above proof of recursion easily extends to show
the corresponding recursions for
$|\mathcal{M}_n^k(Q;c|)|$ and
$|\mathcal{M}_n^k(Q;r|c|)|$.
\end{proof}



\section{Permutations}

In this section we consider classes of permutations that are
enumerated by the poly-Bernoulli numbers resp.~ their relatives. As
usual let $\{1,\ldots, n\} = [n]$ and $S_{n}$ denote the set of
permutations of $[n]$.


\subsection{Vesztergombi permutations}

The permutations we consider in this section are permutations that are
restricted by constraints on the distance between their elements and
their images. The enumeration of such permutation classes is a special
case of a more general problem setting that were investigated by many
authors. Given $n$ subsets $A_1,A_2,\ldots A_n$ of $[n]$, determine
the number of permutations $\pi$ such that $\pi(i)\in A_i$ for all
$i\in[n]$. The problem can be formulated as enumeration of the
$1$--factors of a bipartite graph or as the determination of the
permanent of a $01$ matrix or as the number of rook-placements of a
given board. In general these formulations does note make the problem
easier.

We want to use the results of Lovász and Vesztergombi
(\cite{Vesztergombi},\cite{LovaszVesz}, \cite{Lovasz}) for derivation
of to \eqref{eq:szita} analogous formulas for $C_{n,k}$ and
$D_{n,k}$. We recall definitions and main ideas for the sake of
understanding. Detailed combinatorial proofs and analytical
derivations can be found in the cited articles.

Let $f(r,n,k)$ denote the number of permutations $\pi\in S_{n+k}$
satisfying
\[
-(k+r)< \pi(i)-i < n+r.
\]
The main result is as follows: 


\begin{theorem}\cite{Vesztergombi}
\[
f(r,n,k) = \sum_{m=0}^n(-1)^{n+m}(m+r)!(m+r)^k \stirling{n+1}{m+1}.
\]
\end{theorem}

The original proof is analytic and depends on the solution of certain
differential equation for a generating function based on the
$f(r,n,k)$ numbers. The differential equations capture the recursions
that follow from the expanding rules of the corresponding permanent.

Launois \cite{Launois} realized the connection of this formula to the
poly-Bernoulli numbers, namely that $f(2,n,k)=B_{n,k}$.

\begin{theorem}\cite{Launois}
Let $\mathcal{V}_n^k$ denote the set of permutations $\pi$ of $[n+k]$
such that
$-k\leq \pi(i)-i \leq n\text{ for all }i\in[n+k]$. 
\[|\mathcal{V}_n^k|=B_{n,k}\]
\end{theorem}

Beyond the analytical derivation of the formula there are
combinatorial proofs of the theorem in the literature.  In
\cite{Krotov} the authors define an explicit bijection between
Vesztergombi permutations and lonesum matrices. In \cite{Lovasz} we
find a combinatorial proof for a general case that includes the
theorem.  For the sake of completeness we present here the direct
combinatorial proof from \cite{BB}.

\begin{proof} (Theorem 9.)
$|\mathcal V_n^{k}|$ is the permanent of the $(n+k)\times(n+k)$
  matrix $A=(a_{ij})$, where
\[
a_{ij}=\left\{\begin{array}{cl}
1\quad& \mbox{if} \quad -k\leq i-j\leq n,\quad i=1,\ldots,n+k\\
0\quad&\mbox{otherwise}.
\end{array}\right.
\]
The permanent of $A$ (denoted by $\mbox{per}A$) counts the number of
expansion terms of the matrix $A$ which do not contain a $0$ term.

The matrix $A$ is built up of 4 blocks:
\[
A=\left[
\begin{array}{cc}
J_{n,k} & B_n\\
B^k & J_{k,n}    
\end{array}
\right],
\]
where $J_{n,k}\in\set{0,1}^{n\times k}$, 
$J_{k,n}\in\set{0,1}^{k\times n}$
the matrices with all entries equal $1$, furthermore 
$B_n\in\set{0,1}^{nt\times n}$:
$B_n(i,j)=1$ iff $i\geq j$ and
$B^k\in\set{0,1}^{k\times k}$:
$B^k_{ij}=1$ iff $i\leq j$.

For a term in the expansion of the permanent we have to select exactly
one $1$ from each row and each column. The number of ways of selecting
$1$s from the triangular matrices is given by the Stirling number of
the second kind. (See proof for instance in \cite{Lovasz}.)  So if a
term contains $m$ $1$'s from the upper left block $J_{k,n}$ ($m!$
ways), then it contains $n-m$ $1$'s from $B_n$ ($\stirling{n+1}{m+1}$
ways); $m$ $1$'s from the lower right block ($m!$ ways) and finally
$k-m$ $1$'s from $B^k$ ($\stirling{k+1}{m+1}$ ways).
The total number of terms in the expansion of $\text{per}A$ is
\[
\sum_{m=1}m!\stirling{n+1}{m+1}m!\stirling{k+1}{m+1}.
\]
This proves the theorem.
\end{proof}

Suitable modifications of the definition of Vesztergombi permutations
lead to the pB--relatives.  Let $\mathcal{V}_n^{k*}$ the set of
permutations $\pi$ of $[n+k]$ such that
\[
-k\leq \pi(i)-i < n\quad \mbox{for all} \quad i\in[n+k]
\] 
and  $\mathcal{V}_n^{k**}$ the 
set of permutations $\pi$ of $[n+k]$ such that
\[
-k< \pi(i)-i < n\quad \mbox{for all} \quad i\in[n+k].
\] 

\begin{theorem}(\cite{Vesztergombi})
\begin{itemize}
\item[(i)]
\[
|\mathcal{V}_n^{k*}|=C_{n,k},
\]
\item[(ii)]
\[
|\mathcal{V}_n^{k**}|=D_{n,k}.
\]
\end{itemize}
\end{theorem}

\begin{proof}
In these cases the blocks $B_n, B^k$ are slightly changed. The
modifications are straightforward and hence the details are omitted.
\end{proof}

\begin{corollary}
\begin{itemize}
\item[(i)]
\[
C_{n,k}= \sum_{m=0}^n(-1)^{n+m}m!(m+1)^k\stirling{n+1}{m+1},
\]
\item[(ii)]
\[
D_{n,k}= \sum_{m=0}^n (-1)^{n+m}m!m^k\stirling{n+1}{m+1}.
\]
\end{itemize}
\end{corollary}

\begin{proof}
Clearly $|\mathcal{V}_n^{k*}|=f(1,n,k-1)$ and $|\mathcal{V}_n^{k**}|=f(0,n,k)$.
\end{proof}

In \cite{LovaszVesz} Theorem 1 describes the asymptotic behavior of $D_{n,n}$:.

\begin{theorem}\cite{LovaszVesz}
\[
D_{n,n}\sim\sqrt{\frac{1}{2\pi(1-\ln 2)}}(n!)^2\frac{1}{(\ln 2)^{2n}}.
\] 
\end{theorem}


\subsection{Permutations with excedance set $[k]$}

Permutations that have special restrictions on their excedance set are
enumerated by the poly-Bernoulli numbers resp.  their relatives.  We
note that the connection of this class of permutations to
poly-Bernoulli numbers is not mentioned directly in the literature.

We call an index $i$ an \emph{excedance} (resp.~\emph{weak excedance})
of the permutation $\pi$ when $\pi(i)> i$ (resp.~$\pi(i)\geq i$).
According that we define the set of excedances (resp. the set of weak
excedances) of a permutation $\pi$ as $E(\pi):=\{i|\pi(i)> i\}$ and
$WE(\pi):= \{i|\pi(i)\geq i\}$. Further let define the following sets 
of permutations of $[n+k]$ with conditions on their excedance sets: 
\begin{align*}
\mathcal{E}_n^k&:=
   \{ \pi|\pi \in S_{n+k}\quad \mbox{and} \quad WE(\pi)=[k]\},\\
\mathcal{E}_n^{k*}&:=
       \{\pi|\pi \in S_{n+k}\quad\mbox{and} \quad E(\pi)=[k]\},\\
\mathcal{E}_n^{k**}&:=
       \{\pi|\pi\in S_{n+k}\quad \mbox{and}\quad E(\pi)=[k] 
       \quad\mbox{and}\quad
       \pi(i)\not= i\quad \forall 1\leq i\leq n+k\}.
\end{align*}

The main result in this line of research is summarized in the next
theorem.

\begin{theorem}The following three statements hold:
\begin{itemize}
\item[(i)]
\[
|\mathcal{E}_n^k|=B_{n,k},
\]
\item[(ii)]
\[
|\mathcal{E}_n^{k*}|=C_{n,k},
\]
\item[(iii)]
\[
|\mathcal{E}_n^{k**}|=D_{n,k}.
\]
\end{itemize}
\end{theorem}

\begin{proof}
There are trivial bijections between these permutations and the three
variants of Vesztergombi permutations. We obtain the underlying
matrices of the permutation classes $\mathcal{E}_n^k$,
$\mathcal{E}_n^{k*}$ $\mathcal{E}_n^{k**}$ by shifting the building
blocks of the underlying matrix $A$ of the appropriate variant of the
Vesztergombi permutation. We just sketch the necessary
ideas for (i).
The matrix which permanent
determines the size of this permutation class is built up of the
following $4$ blocks:

\[
E=\left[\begin{array}{cc}
B_k & J_{k,n}\\
J_{n,k} & B^n
         \end{array}\right], 
\]
where $J_{n,k}\in\set{0,1}^{n\times k}$ and
$J_{k,n}\in\set{0,1}^{k\times n}$ are above
(the all-$1$ matrices) and $B_k\in\set{0,1}^{k\times k}$:
$B_k(i,j)=1$ iff $1\leq j\leq i\leq k$,
resp.~$B^n\in\set{0,1}^{n\times n}$: 
$B^n(i,j)=1$ iff $1\leq i\leq j\leq n$.

The terms in the expansion of $\text{per}E$ can be bijectively
identified with the term in the corresponding expansion in the case of
Vesztergombi permutations.
\end{proof}

Next we connect $\mathcal{E}_n^k$ in another way to the poly-Bernoulli
family, hence we give an alternative proof
of (i).  

As we mentioned before permutation tableaux are well studied
objects and several bijections are known between permutations and
permutation tableaux.  We describe a bijection between permutation
tableaux and permutations that is a bijection between the sets
$\mathcal M_n^k(P)$ and $\mathcal{E}_n^k$ when we apply it to the subset of
rectangular Ferrers shapes. 

\begin{theorem}
\[
|\mathcal M_n^k(P)|=|\mathcal{E}_n^k|.
\]
\end{theorem}

We modify the bijection given in
\cite{Burstein} in order to have the following properties: the
excedances of the permutation correspond to the column labels and
fixed points of the permutation to the labels of empty rows. These
modifications do not change the bijection essentially.

\begin{proof}(Sketch)
Consider a $n\times k$ $01$ matrix that avoids the submatrices in
the set $P$ and contains a $1$ in any column.
We assign a permutation to this matrix the following way:
 
Label the positions of the rows from left
to right by $[k]$, the positions of the
columns from bottom to top by $[n]$. We define
the zig-zag path by bouncing right or down every time we hit a
$1$. For $i$ we find $\pi(i)$ by starting at the top of the column $i$
(the left of the row $i$) following the zig-zag path until to boundary
where we hit the row or column labeled by $j$ and set $\pi(i)=j$.

The defined map gives a bijection between the two sets in
the theorem. The details are straightforward and left to the reader.
\end{proof}

In \cite{Lundberg} the authors determined the asymptotic of $C_{n,n}$
investigated as the number of the extremal excedance set statistic.

\begin{theorem}\cite{Lundberg}
\[
C_{n,n}\sim
\left(\frac{1}{2\log 2\sqrt{(1-\log 2)}}+o(1)\right)
\left(\frac{1}{2\log 2}\right)^{2n}(2n)!.
\]
\end{theorem}


\subsection{Callan permutations}

Callan gave an alternative description
of poly-Bernoulli numbers in a note in OEIS \cite{OEIS}.
We repeat it, sketch the proof of his claim.
We do this because Callan permutations play
an important role in proving combinatorial 
properties of pB-relatives.

\begin{defin}
Callan permutations
are the permutations of $[n+k]$ in which each substring whose support
belongs to $N=\{1,2,\ldots, n\}$ or $K=\{n+1, n+2, \ldots, n+k\}$ is
increasing. 
\end{defin}

We call the elements in $N$ \emph{left value elements} and
that of $K$ \emph{right- value elements} and for the sake of
convenience we rewrite 
$K\equiv \{\mathbf{1},\mathbf{2},\ldots, \mathbf{k}\}$
($N=\{1,2,\ldots ,n\}$).
Actually we need just the distinction between the
elements of the sets $N$ and $K$
and an order in $N$ and $K$.
So one can work with $N=\set{0,2,3,\ldots,n}$
and talk about Callan permutations.

Let $\mathcal C_n^k$ denote the set of Callan permutations.

\begin{theorem}
\[
|\mathcal C_n^k|=
\sum_{m=0}^{\min(n,k)}(m!)^2\stirling{n+1}{m+1}\stirling{k+1}{m+1}= B_{n,k}.
\]
\end{theorem}

\begin{proof}(Sketch)
Let $\pi\in\mathcal C_n^{k}$.
Let $\widetilde\pi=0\pi({\bf k+1})$, where $0$ is a new left value 
and ${\bf k+1}$ is a new right value.
Divide $\widetilde\pi$ into maximal blocks of consecutive elements
such a way that each block is a subset of $\set{0}\cup N$
(left blocks)
or a subset of $K\cup\set{\bf{k+1}}$ (right blocks).
The partition starts with a left block
(the block of $0$) and ends with a right block
(the block of ${\bf k+1}$).
So the left and right block alternate, their number is the same,
say $m+1$.
Describing a Callan permutation is 
equivalent to
specifying $m$, a partition $\Pi_{\widehat N}$ of $\widehat N=\set{0}\cup N$
into $m+1$ classes (one class is the class
of $0$, the other $m$ ones are called ordinary
classes), a partition $\Pi_{\widehat K}$  of $\widehat K=K\dot\cup\set{\bf{k+1}}$
into $m+1$ classes
($m$ many of them not containing ${\bf k+1}$,
the ordinary classes), and two orderings of the ordinary classes.
This proves Callan's claim.
\end{proof}

The role of $0$ and ${\bf k+1}$ were
important. With the help of them we had the information
how the left and right blocks follow each other.

Let $\mathcal{C}_n^k(*,l)$ be the set of Callan permutations of $N\dot\cup K$
that end with a left-value element (and hence with a left block). The star
is to remind the reader that there is no
assumption on the leading block of our permutation.
Similarly
let $\mathcal{C}_n^k(l,*)$ be the set of Callan permutations of $N\dot\cup K$
that start with a left-value element.
Let $\mathcal{C}_n^k(l,r)$ be the set of Callan permutations of $N\dot\cup K$
that start with a left-value element and ends with a right element.
The reader easily can define the sets
$\mathcal{C}_n^k(r,*)$, $\mathcal{C}_n^k(*,r)$, $\mathcal{C}_n^k(r,l)$, 
$\mathcal{C}_n^k(l,l)$, and $\mathcal{C}_n^k(r,r)$. 

If we take a Callan permutation, reverse the order of
its blocks (leaving the order within each block)
we obtain a Callan permutation too.
This simple observation proves the following equalities:
\[
|\mathcal{C}_n^k(*,l)|=|\mathcal{C}_n^k(l,*)|,
\]
\[
|\mathcal{C}_n^k(r,l)|=|\mathcal{C}_n^k(l,r)|.
\]

Now we state our next theorem that, gives
a new interpretation of
pB-relatives with the help of
Callan permutations.

\begin{theorem}
\begin{itemize}
\item[(i)]
\[
C_{n,k}=|\mathcal{C}_n^k(*,l)| = 
\sum_{m=0}^{\min(n,k)}(m!)^2\stirling{n+1}{m+1}\stirling{k}{m}
\quad n\geq 1\text{ and }k\geq 0.
\]
\item[(ii)]
\[
D_{n,k}=|\mathcal{C}_n^k(l,r)| =
\sum_{m=0}^{\min(n,k)}(m!)^2\stirling{n}{m}\stirling{k}{m}
\quad n\geq 1\text{ and }k\geq 1.
\]
\end{itemize}
\end{theorem}

\begin{proof}
(i):
Take a $\pi\in\mathcal{C}_n^k(*,l)$ and extend it 
with a starting $0$ (an extra left value): $\widehat\pi=0\pi$.
One extra element is enough to control the
structure of blocks:
the block decomposition
starts and ends with a left block.
Let $m+1$ the number of left blocks, $m$ is the number of right blocks
(and the number of ordinary left blocks, i.e. blocks not containing $0$).
The rest of the proof is a straightforward modification of the previous one.

(ii): Without any extra element we control the starting and
ending block. $m$ denotes the common number of
left and right blocks. The details are left to the reader.
\end{proof}

The next lemma is implicit in \cite{BH}.
Since it is central for us, we present it here.

\begin{lemma}
There is a bijection
\[
\varphi:\mathcal C_n^k(*,l)\to\mathcal C_{n-1}^{k+1}(*,r).
\]
\end{lemma}

\begin{proof}
Take any $\pi\in \mathcal C_n^k(*,l)$.
Find $n$ (the largest left value) in it.
It is the last element of one of the left blocks 
(possibly the very last element
of $\pi$).

Assume that $n$ is not the last element of $\pi$.
Then it is followed by a right block $R$ and by at least one left block.
Exchange $n$ to ${\bf k+1}$ and move $R$ to the end of $\pi$. 
The permutation that we obtain this way will be $\varphi(\pi)$.

If $n$ is the last element of $\pi$, then exchange it to 
${\bf k+1}$. 

In both case the described image is obviously in $C_{n-1}^{k+1}(*,r)$.
In order to see that $\varphi$ is a bijection we need to 
construct its inverse. This can be done easily
based on ${\bf k+1}$. The details are left to the reader.
\end{proof}

The lemma is obviously gives us a 
$\psi:\mathcal C_n^k(*,r)\to\mathcal C_{n+1}^{k-1}(*,l)$
bijection too.

In \cite{BH} this lemma was used to prove that
$\sum_{k,\ell:k+\ell=n}(-1)^kB_{k,\ell}=0$.
We use the lemma for different purposes.
First we combinatorially prove the symmetry of the $C_{n,k}$ 
numbers.

\begin{corollary}
\[
C_{n,k}=C_{k+1,n-1}.
\]
\end{corollary}

\begin{proof}
Change the role of left and right values.
The two orderings remain, hence we obtain a Callan permutation
(the blocks remain the same). This leads to a bijection
between $\mathcal C_n^k(*,l)$ and
$\mathcal C_k^n(*,r)$. Using the previous lemma we obtain that
\[
C_{n,k}=|\mathcal C_n^k(*,l)|=|\mathcal C_k^n(*,r)|=
|\mathcal C_{k+1}^{n-1}(*,l)|=C_{k+1,n-1}.
\]
\end{proof}
 
The next application of our lemma will be a simple
connection between poly-Bernoulli numbers
and its C-relative. It was proved in \cite{KanekoLast}
with analytical methods. Here we present a combinatorial
method.

\begin{theorem}\cite{KanekoLast}
\[
B_{n,k}=C_{n,k}+C_{k,n}=C_{n,k}+C_{n+1,k-1}.
\]
\end{theorem}

\begin{proof}
We know that $B_{n,k}=|\mathcal C_n^k|$,
furthermore $\mathcal C_n^k=\mathcal C_n^k(*,l)\dot\cup
\mathcal C_n^k(*,r)$.
We have a bijection between
$\mathcal C_n^k(*,r)$ and $\mathcal C_{n+1}^{k-1}(*,l)$.
Hence
\[
B_{n,k}
=|\mathcal C_n^k|
=|\mathcal C_n^k(*,l)|+|\mathcal C_n^k(*,r)|
=C_{n,k}+|\mathcal C_{n+1}^{k-1}(*,l)|
=C_{n,k}+C_{n+1,k-1}.
\] 
\end{proof}

A similar connection is true between $C_{n,k}$ and $D_{n,k}$. 

\begin{theorem}
\[
C_{n,k}=D_{n,k}+D_{n-1,k}+D_{n-1,k+1}.
\]
\end{theorem}

\begin{proof}
We know that $C_{n,k}=|\mathcal C_n^k(*,l)|$,
furthermore $\mathcal C_n^k(*,l)=\mathcal C_n^k(r,l)\dot\cup
\mathcal C_n^k(l,l)$.
\[
C_{n,k}
=|\mathcal C_n^k(*,l)|
=|\mathcal C_n^k(r,l)|+|\mathcal C_n^k(l,l)|
=D_{n,k}+|\mathcal C_n^k(l,l)|.
\] 

The second term can be handled as we handled
$\mathcal C_n^k(*,l)$ on our lemma:
We present a bijection
$\varphi:\mathcal C_n^k(l,l)\to
\mathcal C_{n-1}^{k+1}(l,r)
\dot\cup
\mathcal C_{n-1}^k(r,l).$

Let $\pi\in \mathcal C_n^k(l,l)$.
Find the position of $n$ (the largest left value) in $\pi$.
It is the last element of one of the
left blocks. If it is in the last block
then simply rewrite it to ${\bf k+1}$.
If it is not in the last block then there is 
a following right block $R$ and at least one more
left block.
Then also rewrite it to ${\bf k+1}$ and at the same time move $R$
to the end of $\pi$. The resulting permutation
is $\varphi(\pi)$.

So far we did the same as we did in the proof of the lemma.
The only problem, that the image is not necessarily
in $\mathcal C_{n-1}^{k+1}(l,r)$. It is possible that the block
of $n$ is the first block of $\pi$ and it consists of only one element.
Then the lemmas idea leads to $\varphi(\pi)$
where the leading element is ${\bf k+1}$. We don't want that. 
In this very special case ($n$ is the first element of $\pi$)
we just erase $n$ from $\pi$ in order to obtain
$\varphi(\pi)$.

Now it is clear that we defined a map with
$\mathcal C_{n-1}^{k+1}(l,r)
\dot\cup
\mathcal C_{n-1}^k(r,l)$ as codomain.
To see that it is a bijection we construct its inverse:
If we have a permutation from $C_{n-1}^{k}(r,l)$, then the
inverse puts a starting $n$ in front of it.
If we have a permutation from $C_{n-1}^{k+1}(l,r)$, then the
inverse works as in our lemma.

The bijection leads to a fast end to our proof:
\[
C_{n,k}
=D_{n,k}+|\mathcal C_n^k(l,l)|
=D_{n,k}+|\mathcal C_{n-1}^k(r,l)|+|\mathcal C_{n-1}^{k+1}(l,r)|
=D_{n,k}+D_{n-1,k}+D_{n-1,k+1}.
\] 
\end{proof}






\section{Acyclic orientations of bipartite complete graphs}

The connection of poly-Bernoulli numbers to acyclic orientations of
the bipartite complete graph was discovered independently in two lines
of research. (\emph{Acyclic orientation} of a graph is an assignment
of direction to each edge of the graph such that there are no directed
cycles.) 

Cameron, Glass and Schumacher \cite{Cameron} investigated the
problem of maximizing number of acyclic orientations of
graphs with $v$ vertices and $e$ edges. They
conjecture that if $v=2n$ and $e=n^2$ then $K_{n,n}$ is
the extremal graph. Along their research they
counted the acyclic orientations of $K_{n,k}$, and established a 
bijection between these orientations and 
lonesum matrices of size $n\times k$.

In \cite{EhWi} 
the authors realized the connection of the permutations
with extremal excedance sets (see section 4) 
and acyclic orientations with a unique
sink. Without referring to the C-relatives of poly-Bernoulli numbers 
they gave an interpretation of the $C_{n,k}$ numbers in terms
of acyclic orientations of complete bipartite graphs.
Their proof is a specialization of general statements,
we reprove the version, we need, by elementary means.

We extend their results with an interpretation
for $D_{n,k}$ and summarize this line of research in the next theorem.
We need some notation. Let $N=\set{u_1,u_2,\ldots,u_n}$,
$\widehat N=N\cup\set{u}$,
$M=\set{v_1,v_2,\ldots,v_k}$,
$\widehat M=M\cup\set{v}$ be vertex sets.
Let $K_{A,B}$ denote the complete bipartite
graphs on $A\dot\cup B$.
Let $\mathcal D_n^k$ denote
the set of acyclic orientations of $K_{N,M}$.
Let $\mathcal D_n^k{'}$ denote
the set of acyclic orientations of $K_{N,\widehat M}$, where
$v$ is the only sink (vertex without outgoing edge).
Let $\mathcal D_n^k{''}$ denote
the set of acyclic orientations of $K_{\widehat N,\widehat M}$,
where $u$ is the only source (vertex without ingoing edge)
and $v$ is the only sink.

\begin{theorem}
\begin{itemize}
\item[(i)]\cite{Cameron}
\[
|\mathcal D_n^k|=B_{n,k},
\]
\item[(ii)]\cite{EhWi} 
\[
|\mathcal D_n^k{'}|=C_{n,k},
\]
\item[(iii)]
\[
|\mathcal D_n^k{''}|=D_{n,k}.
\]
\end{itemize}
\end{theorem}

\begin{proof}
(i)\cite{Cameron}: 
An acyclic orientation of
$K_{N,K}$ can be coded by a $01$ matrix $B$ of size $n\times k$
the following way: $b_{i,j}=0$ whenever the
edge $u_iv_j$ is oriented from $u_i$ to $v_j$, and $b_{i,j}=1$ whenever the
edge $u_iv_j$ is oriented from $v_j$ to $u_i$.  It is
easy to check that the orientation is acyclic
iff the corresponding matrix $B$ does not contain any of the submatrix
of the set $L$, hence $B$ lonesome. This establishes a bijection between
$\mathcal D_n^k$ and $\mathcal L_n^k$. The claim follows from our previous
results.

(ii):
Take a binary matrix coding an orientation of a complete
bipartite graph $K_{A,B}$. An all-$0$ column (the column of vertex
$w\in B$) corresponds
to the information that $w$ is a sink.
Hence if we take an arbitrary orientation
of $K_{N,\widehat M}$ from $\mathcal D_n^k{'}$, 
then its restriction to $K_{N,M}$ will be
an acyclic orientation. Its coding binary matrix
cannot contain an all-$0$ column (an all-$0$ column would
correspond to a second sink, that cannot exists
in an orientation from $\mathcal D_n^k{'}$).
Note that there is no restriction on rows. The elements
of $N$ cannot be sinks, since the edges connecting them to
$v$ are outgoing edges.

The above argument gave us
a bijection between $\mathcal D_n^k{'}$ and $\mathcal L_n^k(c|)$,
hence it proves our claim.

(iii):
Straight forward extension of the previous proof.
\end{proof}

Using classical results the connection to acyclic orientation
of complete bipartite graphs 
immediately leads
to connections to the chromatic polynomials
of complete bipartite graphs.
The chromatic polynomial of a graph $G$ is
a polynomial $\text{chr}_G(q)$, such that
for natural number $k$
$\text{chr}_G(k)$ gives the number of
good $k$-colorings of $G$.

The famous result
of Stanley \cite{Stan} is that the number of the acyclic orientations
of a graph is equal to the absolute value of the chromatic polynomial
of the graph evaluated at $-1$. Green and Zaslavsky \cite{Green}
showed that the number of acyclic orientations with a given unique sink is
(up to sign) the coefficient of the linear term of the chromatic
polynomial (see  \cite{Geb} for elementary proofs).
Also in \cite{Green}
it is proven, that the
number of acyclic orientations of a graph
$G$ with a specified $uv$ edge, such that $u$ is 
the unique source and $v$ is the unique sink is
the derivative of the chromatic polynomial evaluated at $1$
(the necessary signing is taken), the so called ``Crapo's beta invariant''.
Again \cite{Geb} present an elementary discussion
of this result.

By putting together the information quoted above we obtain
the following theorem.

\begin{theorem}
\begin{itemize}
\item[(i)]\cite{Cameron}
\[
B_{n,k}=(-1)^{n+k}\text{chr}_{K_{n,k}}(-1),
\]
\item[(ii)]\cite{EhWi} 
\[
C_{n,k}=(-1)^{n+k}[q]\text{chr}_{K_{n,k+1}}(q),
\]
\item[(iii)]
\[
D_{n,k}=(-1)^{n+k}
\left(
\frac{d}{dq}\text{chr}_{K_{n+1,k+1}}
\right)
(1).
\]
\end{itemize}
\end{theorem} 

The chromatic polynomials of complete 
bipartite graphs are well understood.
We list a few results on this topic.

The
exponential generating function of the chromatic polynomial of
$K_{n,k}$ \cite{StanEc2} Ex. 5.6:
\[
\sum_{n\geq 0}\sum_{k\geq 0}\text{chr}_{K_{n,k}}(q)\cdot
\frac{x^n}{n!}\frac{y^k}{k!}=
(e^x+e^y-1)^q
\]

Several formulas for
the chromatic polynomial of complete bipartite graphs are known
(for example \cite{Swen}, \cite{EhWi}, \cite{Hubai}):

\[
\text{chr}_{K_{n,k}}(q)=\sum_{i=0}^n\sum_{j=0}^k 
\stirling{n}{i}\stirling{k}{j}(q)_{i+j}, 
\]
where $(q)_\ell =q(q-1)(q-2)\ldots(q-\ell+1)$, the ``falling factorial''.
\[
\text{chr}_{K_{n,k}}(q)=\sum_{m\geq 0} \left(\sum_{i=0}^n
\sum_{j=0}^k s(i+j,m)\stirling{n}{i}\stirling{k}{j}\right)q^m, 
\]
where $s(n,k)$ is the (signed) Stirling number of the first kind.

Simple arithmetic leads to the following theorem:

\begin{theorem}
\begin{itemize}
\item[(i)]
\[
B_{n,k}=(-1)^{n+k}\sum_{m=0}^{n+k}\sum_{i=0}^n\sum_{j=0}^k(-1)^m s(i+j, l)
\stirling{n}{i}\stirling{k}{j},
\]
\item[(ii)]
\[
C_{n,k}=(-1)^{n+k}\sum_{i=0}^n\sum_{j=0}^k s(i+j, 1)
\stirling{n}{i}\stirling{k+1}{j},
\]
\item[(iii)]
\[
D_{n,k}=(-1)^{n+k}\sum_{l=0}^{n+k+2}\sum_{i=0}^n\sum_{j=0}^k ls(i+j, l)
\stirling{n+1}{i}\stirling{k+1}{j}.
\]
\end{itemize}
\end{theorem}

We mention that the formula in (ii) is implicit
in \cite{EhWi}, without mentioning the poly-Bernoulli connection.


\section{Algorithms for generating the series}

In this section we recall algorithms that computes the arrays
$B_{n,k}$, $C_{n,k}$ and $D_{n,k}$ by similar simple rules as Pascal's
triangle the binomial coefficients. We
will see that in this context the relatives $D_{n,k}$ arise naturally.

This line of research was initiated
by the
Akiyama--Tanigawa algorithms, that generates the Bernoulli numbers. 
Let define the array $a_{n,i}$ recursively (based on $\set{a_{0,i}}$) 
by the
rule:
\[
a_{n+1, i}=(i+1)(a_{n,i}-a_{n,i+1}).
\] 
Akiyama--Tanigawa
proved that if the initial sequence is $a_{0,i}=\frac{1}{i}$ then $a_{n,0}$ are
the $n$-th Bernoulli numbers. Let denote by AT the
transformation $\set{a_{0,i}}\to\set{a_{n,0}}$.
Akiyama--Tanigawa's theorem says that $AT(\set{1/(i+1)}_i)=\set{B_i}_i$ (with $B_1=\frac{1}{2}$).

Kaneko showed \cite{KanekoAT} that for
any initial sequence $a_{0,i}$ it holds
\[
a_{n,0}=\sum_{i=0}^n (-1)^ii!\stirling{n+1}{i+1}a_{0,i}.
\]

Based on the sieve formulas
and simple arithmetic we obtain the following theorem.

\begin{theorem}
\begin{itemize}
\item[(i)]\cite{KanekoAT}
\[
AT(\set{(i+1)^k}_i)=\set{(-1)^iC_{i,k}}_i,
\]
\item[(ii)]
\[
AT(\set{i^k}_i)=\set{(-1)^iD_{i,k}}_i
\]
\end{itemize}
\end{theorem}

The poly--Bernoulli numbers itself can be generated also by such simple
rules, though the recursive rule has to be modified for that. 
Chen \cite{Chen} presents the variant of the algorithm with these changes:
\[
b_{n+1,i}= ib_{n,i}-(i+1)b_{n,i+1},
\]
Chen shows (Proposition 2.) that
\[
b_{n,0}= \sum_{i=0}^n(-1)^ni!\stirling{n}{i}b_{0,i}.
\]
Let denote BT transformation $\set{b_{0,i}}\to\set{b_{n,0}}$ the transformation based on the
modified recursive rule. One consequence of Chen's theorem (\cite{Chen} Theorem 1.) is $BT(\set{1/(i+1)}_i)=\set{B_i}_i$ (with $B_1=-\frac{1}{2}$) and another is that
\begin{theorem}
\[BT(\set{(i+1)^k}_i)=\set{(-1)^iB_{i,k}}_i\]
\end{theorem}


\section{Diagonal sum of poly-Bernoulli numbers}

The diagonal sum of poly-Bernoulli numbers resp.~their relatives
arise in analytical, number theoretical and combinatorial
investigations \cite{OEIS}, \cite{KanekoLast}, \cite{Lundberg}.
However a nice formula is still missing.  The diagonal sum of the
poly-Bernoulli numbers
\[
\sum_{n+k=N}B_{n,k}
\] 
are referred in OEIS \cite{OEIS} A098830: \[1,2,4,10,32,126, 588,
3170,\ldots\] The diagonal sum of the $C$-relatives are
also referred in OEIS \cite{OEIS} A136127:
\[1,2,5,16,63,294,1585\ldots\] The simple arithmetic relation between $B_{n,k}$ and
$C_{n,k}$ of theorem 17. implies that (except the first entry) the
A098830 is exactly the double of A136127.

From the combinatorial point of view the diagonal sum enumerates of
course sets of the combinatorial objects we listed in this paper
before. However there are combinatorial objects where this sum itself
appears naturally: there is no reason for the division of the basic set
 of size $N$ into two sets of size $n$ and $k$ with $n+k=N$. Here we mention some of them.

 The ascending-to-max property \cite{HeMuRa} is one of the characteristic property of
permutations that are suffix arrays of binary words. Suffix arrays
play an important role in efficient searching algorithms of given
patterns in a text. 

Cycles without stretching pairs \cite{Lundberg} received attention because of their
connection to a result of Sharkovsky in discrete dynamical systems. 
The occurrence of a stretching pair within a periodic orbit implies turbulence \cite{Lundberggen}. 
In \cite{Lundberggen} we find also the description of strong connections to permutations 
that avoid $(21-34)$ or $(34-21)$ as generalized patterns. 
 
The introduction of the
combinatorial non-ambiguous trees \cite{Aval} that are compact embeddings of
binary trees in a grid, was motivated by enumeration of parallelogram
polynomios. Non-ambiguous trees are actually special cases of tree-like tableaux,
 objects that are in one-to-one correspondence with permutation tableaux.

 From the analytical results we recall here an interesting connection
to the central binomial sum:
\begin{align*}
CB(k)=\sum_{n\geq 1}\frac{n^k}{\binom{2n}{n}}
\end{align*}
Borwein and Girgensohn (\cite{Borwein} section 2.) showed that
\[
CB(N)=P_N + Q_N\frac{\pi}{\sqrt{3}},
\]
where $P_N$ and $Q_N$  are explicitly given rationals. 

Stephan's computations suggest the following interesting conjecture \cite{OEIS}:

\begin{conjecture}
\[
\sum_{n+k = N}B_{n,k} =3P_N.
\]
\end{conjecture}

Based on the explicit formula that was given in  \cite{Borwein} we can reformulate the
Stephan's conjecture:

\[
\sum_{n+k = N}B_{n,k}=
3P_N=(-1)^{N+1}\frac{1}{2}\sum_{j=1}^{N+1}(-1)^j
j!\stirling{N+1}{j}\frac{\binom{2j}{j}}{3^{j-1}} \sum_{i=0}^{j-1}\frac{3^i}{(2i+1)
\binom{2i}{i}}.
\]

It would be interesting to prove the conjecture or/and to find a simple
expression for the diagonal sum.




  

\end{document}